\documentclass[12pt]{amsart}
\usepackage{amscd,amssymb}
\usepackage[graph,frame,poly,arc]{xy}  
\usepackage[plainpages,backref,urlcolor=blue]{hyperref}

\theoremstyle{plain}
\newtheorem{theorem}{Theorem}[section]
\newtheorem{corollary}[theorem]{Corollary}
\newtheorem{lemma}[theorem]{Lemma}

\theoremstyle{definition}
\newtheorem{definition}[theorem]{Definition}
\newtheorem{remark}[theorem]{Remark}
\newtheorem{example}[theorem]{Example}
\newtheorem{proposition}[theorem]{Proposition}

\numberwithin{equation}{section}
\newcommand{\A}{\mathcal{A}}
\newcommand{\C}{\mathbb{C}}
\newcommand{\LL}{\mathcal{L}}

\newcommand{\PP}{\mathbb{P}}
\newcommand{\Q}{\mathbb{Q}}

\newcommand{\F}{\mathcal{F}}
\newcommand{\OO}{\mathcal{O}}
\newcommand{\HH}{\mathbb{H}}

\newcommand{\Bl}{\operatorname{\it B\ell}}

\date{\today}
\begin{document}

\title [A vanishing result for twisted cohomology] {A vanishing result for the first twisted cohomology of affine varieties and applications to line arrangements}
%\author{Pauline Bailet, Alexandru Dimca, and Masahiko Yoshinaga}
\author[Pauline Bailet]{Pauline Bailet}
\address{Department of Mathematics \\
Bremen University,  Bremen, Germany} 
\email{pauline.bailet@uni-bremen.de}
\author[Alexandru Dimca]{Alexandru Dimca}
\address{Universit\'e C\^ ote d'Azur, CNRS, LJAD, France }
\email{dimca@unice.fr}

\author[Masahiko Yoshinaga]{Masahiko Yoshinaga}
\address{Hokkaido University, Sapporo, 060-0810, JAPAN}
\email{yoshinaga@math.sci.hokudai.ac.jp}



\begin{abstract} 
A general vanishing result for the first cohomology group of affine smooth complex varieties
with values in rank one local systems is established. This is applied to the determination of the monodromy action on the first cohomology group of the Milnor fiber of some line arrangements,
including the monomial arrangement and the exceptional reflection arrangement of type $G_{31}$.
\end{abstract}

\maketitle

\section{Introduction}

In this paper we prove a general vanishing result for the first cohomology group of affine smooth complex varieties
with values in some rank one local systems, see Theorem \ref{thm:vanish}. This result can be regarded as a generalization of the main result in \cite{cdo}. To prove this result, and especially to use it in concrete cases, one needs a new ingredient, namely  characterizations of affine smooth surfaces such as  Nakai-Moishezon criterion, recalled in Proposition \ref{thmNM},
and Nagata's Theorem, recalled in Proposition \ref{Nagata}. These results on affine surfaces are given in section 2, as well as some related results, e.g. Proposition \ref{propconn}.

The main application motivating Theorem \ref{thm:vanish} is the study of the cohomology of line arrangement complements $M(\A)$ with values in some rank one local systems, in particular the study of  Milnor fiber monodromy
of such arrangements $\A$, see \cite[Chapter 5]{DHA}. In section 3 we introduce the basic notation and prove Proposition \ref{keyprop} which is a useful tool in proving that some non-proper surfaces constructed by blowing-ups from a line arrangement $\A$ in the projective plane $\PP^2$ are affine. Example \ref{Ziegler} shows that such affinity questions are rather subtle, and the answer is not determined by the combinatorics, i.e. by the intersection lattice $L(\A)$.

The vanishing result is proved in section 4 and the main consequence for line arrangements is given in Proposition \ref{prop:general}, which says roughly that if there is a line in the arrangement $\A$ containing just one bad point (i.e. a point in $T_{=1}$ in the notation from section 3)
of the rank one local system $\LL$, then $H^1(M(\A),\LL)=0$ if some very mild extra condition holds. This result was known for line arrangements defined over the real numbers without this mild extra condition,
see \cite{yos-mil, yos-locsys}. A similar result is also known for certain monodromy eigenspaces of the Milnor fiber of a complex line arrangement, see \cite{ba-yo}. All these results, as well as \cite{cdo}, were motivated by Libgober's result in \cite{Li} saying that for the local systems $\LL$ related to the Milnor fiber monodromy, one has such a vanishing as soon as there is a line in the arrangement $\A$ containing no bad point for $\LL$.

We obtain as a direct application of Proposition \ref{prop:general} and using previous results by M\u acinic, Papadima and Popescu in \cite{MPP} a complete description on the Milnor fiber monodromy of the 
monomial arrangements a.k.a. the Ceva arrangements 
$$\A(m,m,3): (x^m-y^m)(y^m-z^m)(x^m-z^m)=0,$$
for $m \geq 3$, see Corollary \ref{cor:Ceva}. This result was previously established in \cite{Dreflection}, using completely different techniques, namely residues of rational differential forms with poles along the line arrangement $\A$.
In the final section we apply the same approach to the exceptional  reflection arrangement $\A(G_{31})$, consisting of 60 planes in $\PP^3$.
The result stated in Proposition \ref{G31} was first established in \cite{DStG31} using again different techniques,  based on a description of the basis of the Jacobian syzygies for the defining equation  of the  arrangement $\A(G_{31})$, given in \cite[Appendix B]{OT} and going back to the work of Maschke in the 19th century on invariants of reflection groups. Our new proof here is based on Theorem \ref{thm:vanish} applied to a generic plane section $\A= H \cap \A(G_{31}) $ of the plane arrangement $\A(G_{31})$ in $\PP^3$. However,
this new proof requires some computer aided computations (though less complex than those needed for the approach in \cite{DStG31}) to decide the position of the 30 points of multiplicity 6 in $\A$ with respect to curves in $H=\PP^2$ of degree $d=1,2,3$ and $4$. The corresponding code in SINGULAR was written for us by Gabriel Sticlaru and is available on request. The fact that the monomial arrangements $\A(m,m,3)$ for $m \geq 3$ and the exceptional  reflection arrangement $\A(G_{31})$ can be treated in a uniform way shows in our opinion the power of this new approach based on the vanishing given by Theorem \ref{thm:vanish}.

\bigskip

We thank Lucian B\u adescu for useful discussions and Gabriel Sticlaru for his help with the SINGULAR code. Pauline Bailet is supported by the University of Bremen and the European Union FP7 COFUND under grant agreement n\textsuperscript{o} 600411. Masahiko Yoshinaga is partially supported by JSPS KAKENHI Grant Number 15KK0144, 16K13741, and Humboldt Foundation.

\section{Affine open subsets of projective smooth surfaces}

In this section, let $S$ be a complex projective smooth surface, $D \subset S$ a reduced divisor and $U=S \smallsetminus D$ the corresponding complement. We are interested in deciding whether $U$ is an affine surface.
We recall first the following result, see \cite[Theorem  V.1.10]{Har} or \cite[Theorem  1.2.6]{Mats}.
\begin{proposition}
(Nakai-Moishezon criterion) 
\label{thmNM}
Let $L$ be a line bundle on $S$. 
Then $L$ is ample if and only if $L^2>0$ and $L\cdot C>0$ for all irreducible 
curves $C\subset S$. 
\end{proposition}

Let $D=\sum_{i=1}^n D_i$ be an irreducible decomposition of the divisor 
$D$ on $S$. The complement  $U$ is affine if and only if it exists  an effective ample divisor
$\widetilde{D}=\sum_{i=1}^n a_i D_i$ ($a_i>0$) 
see \cite[Theorem 2]{Good}. Here and in the sequel we consider divisors with $\Q$-coefficients.
If we want to apply the Nakai-Moishezon criterion to the line bundle $L= \OO(\widetilde{D})$
in the case when the coefficients $a_i$ are integers, we need to check in particular the inequalities
\begin{equation}
\label{eq:1}
D_i\cdot\widetilde{D}>0 \text{ for all }i=1, \dots n.
\end{equation}
\begin{definition}
The divisor $D$ is said to support a \emph{NM divisor} if there exists 
an effective divisor $\widetilde{D}=\sum_{i=1}^n a_i D_i$ ($a_i>0$) 
such that the inequalities \eqref{eq:1} hold. 
\end{definition}

\begin{lemma}
\label{lem:weak}
Let $I\subset [n]=\{1, 2, \dots, n\}$ be a nonempty subset. 
Suppose that $D':=\sum_{i\in I}D_i$ supports an NM divisor, 
and for each $i\in [n]\smallsetminus I$, $D_i\cap D'\neq\emptyset$. 
Then $D$ supports an NM divisor. 
\end{lemma}

\begin{proof}
Suppose $\widetilde{D}'=\sum_{i\in I}a_i D_i$ satisfies the condition. 
Then 
$\widetilde{D}=\widetilde{D}'+\varepsilon\sum_{i\in [n]\smallsetminus I}D_i$ 
satisfies the condition for $0<\varepsilon\ll 1$. 
\end{proof}

\begin{corollary}
\label{cor:connect}
Suppose that the divisor $D$ is connected and $D':=\sum_{i\in I}D_i$ supports 
an NM divisor for some nonempty subset $I\subset [n]$. Then 
$D$ supports an NM divisor. 
\end{corollary}

\begin{proof}
By the assumption, we can choose an ordering $[n]\smallsetminus I
=\{i_1, i_2, \dots, i_k\}$ such that 
$D_{i_\alpha}\cap (D'+D_{i_1}+\dots+D_{i_{\alpha-1}})\neq\emptyset$. 
Apply Lemma \ref{lem:weak} inductively. 
\end{proof}

Note also the following result by Nagata, see for instance  \cite[Corollary 3.3]{Har0}, which in our case can be applied when $U'=S\smallsetminus D'$ is affine.

\begin{proposition} \label{Nagata}
(Nagata's Theorem) Let $U'$ be a complex affine smooth surface and $Y$ be a pure 1-codimension closed subset of $U'$. Then $U=U' \smallsetminus Y$ is also affine.
\end{proposition}

We have also the following result, perhaps well known to the specialists. As we were not able to locate a reference, we include the short proof. Example \ref{Ziegler} below shows that the converse of the last claim fails.

\begin{proposition} \label{propconn}

Let $S$ be a connected projective smooth surface, $D$ a divisor on $S$ and $U=S\smallsetminus D$. Then 

\begin{enumerate}
\item 
$H^4(U)=0$. 
\item 
Suppose $b_1(S)=0$. Then $H^3(U)=0$ if and only if
the divisor  $D$ is connected.
\item 
If $U$ is affine, then $H^3(U)=0$. 
\end{enumerate}

\end{proposition}

\proof The first claim follows since $U$ is a non-compact 4-manifold.
One has the following exact sequence in cohomology with compact coefficients
$$ H^0_c(U) \to H^0(S) \to H^0(D) \to H^1_c(U) \to 0=H^1(S),$$
where the vanishing comes from $b_1(S)=0$. This proves the second claim,
since $S$ is clearly connected. The last claim is well known, e.g. as a special case of results by H. Hamm and/or M. Artin, see for instance \cite[Corollary 5.2.19]{SIT} .

\endproof

\section{Line arrangements and local systems}

Let $\A=\{H_1, \dots, H_n\}$ be a line arrangement on 
the complex projective plane $\PP^2$. 
To avoid trivial exceptions, 
we assume that $\A$ is not of pencil type (equivalently, 
each $H_i$ has at least two intersections of $\A$ on it). 

Let $\LL$ be a rank one local system on the complement 
$M(\A)=\PP^2\smallsetminus\bigcup_{i=1}^n H_i$. 
The isomorphism class of the local system $\LL$ is 
determined by the monodromy $t_i\in\C^\times$ around $H_i$. 
We assume $t_i\neq 1$ for all $i=1, \dots, n$ and we 
note that $\prod_{i=1}^n t_i=1$. For a point $p\in\PP^2$, 
denote by $t_p=\prod_{H_i\ni p}t_i$  the corresponding total turn monodromy of $\LL$, see \cite[Section 5.2]{DHA}.

Let $T\subset\PP^2$ be the set of all intersections of $\A$ with 
multiplicity $\geq 3$. The set $T$ is decomposed into 
$T=T_{\neq 1}(\LL) \sqcup T_{=1}(\LL)$ according to 
%$t_p=\prod_{H_i\ni p}t_i$ is 
$t_p\neq 1$ or $t_p=1$. To keep the notation simple, we set $T_{\neq 1}=T_{\neq 1}(\LL)$ and
$T_{=1}=T_{=1}(\LL)$, but keep in mind that these sets depend on the local system $\LL$.

Denote the set of all double points by $P\subset\PP^2$, and set 
$k_i:=\# H_i\cap T$, 
$k'_i:=\# H_i\cap T_{\neq 1}$, and 
$d_i:=\# H_i\cap P$. 
%, 
%the number of points to blow-up on $H_i$,  
%the number of intersections on $H_i$ 
%such that the multiplicity is $\geq 3$ and $t_q\neq 1$, and 
%the number of double points on $H_i$, respectively. 

Let $S=\Bl_T\PP^2$ be the surface obtained by blowing up the 
points in $T$. We denote the exceptional divisor $E_p$ for $p\in T$. 
We also denote by $\overline{H}_i\subset S$ the strict transform of the line
$H_i$. 

We consider the following divisor on $S$: 
\[
D=\sum_{i=1}^n\overline{H}_i+\sum_{p\in T_{\neq 1}}E_p. 
\]

Note that $D$ is normal crossing, and is the sum of all components 
(of the total transform of $\A$) with nontrivial monodromy around it. 

\begin{proposition} \label{propNM}

(1) 
$D\cdot \overline{H}_i=1-k_i+k'_i+d_i$. 

(2) $D\cdot E_p=-1+n_p\geq 2$, for $p\in T_{\neq 1}$, 
where $n_p$ is the multiplicity of $p$ in $\A$. 
\end{proposition}

\begin{proof}
Straightforward. 
\end{proof}

\begin{example}\label{Hirz}
Let $p_1, \dots, p_\ell\in\PP^2$ 
be distinct points on a line $L$ in the plane $\PP^2$. 
Let $\A_i=\{H_{i,1}, \dots, H_{i,c_i}\}$ be a set of lines passing 
through $p_i$ and $\A=\A_1\sqcup\cdots\sqcup\A_\ell$. 
We assume that $L\notin \A_i$ and all intersections 
of $\A$ except for $p_1, \dots, p_\ell$ are double points. 
Suppose $t_{p_1}=\cdots=t_{p_\ell}=1$. 
Then the strict transform $\overline{L}$ does not intersect $D$, $D\cdot L=0$, 
and hence $U=S\smallsetminus D$ is not affine. 

(Note that if $t_i\neq 1$ for all lines, then $H^1(M, \LL)=0$. It is proved 
as follows. Firstly, note that the local system cohomology group $H^1(M, \LL)$ 
is a homotopy invariant. Therefore, $\dim H^1(M, \LL)$ is constant 
during the lattice isotopic deformation \cite{Ra}. By the result in 
\cite{NaYo}, the moduli space of such line arrangements is irreducible, hence 
connected. One can deform it to a real arrangement. Then applying 
\cite[Theorem 3]{yos-locsys}, one obtains $H^1(M, \LL)=0$.) 
\end{example}

%\begin{example}
%We can generalize the above example. Let $L$ be a line, 
%$p_1, \dots, p_k\in L$. Consider pencils of each point $p_i$ and 
%get $\A$ by collecting them together (exclude $L$). 
%\end{example}
The following result controls to a certain extent the existence of bad curves as in the above example.

\begin{proposition} \label{keyprop}

Let $C\subset\PP^2$ be an irreducible curve of degree $d$ such that  its strict transform in $S=\Bl_T\PP^2$ does not meet $D$.
For each point $p \in T_{= 1}$ denote by $m_p \geq 0$ the multiplicity of  $C$ at $p$. Then the following holds.
\begin{enumerate}

\item The intersection $C\cap\left(\bigcup_{H_i\in\A}H_i\right)$ is transverse and produces a subset of $T_{= 1}$. For each $j=1,\dots, n$, one has $\sum_{p \in H_j \cap T_{= 1}}m_p=d$. In particular, if there is a line $H_j$ containing only one point $p$ of $C \cap T_{= 1}$, then $C$ is a line.

\item If $n_p$ is the multiplicity of the point $p$ in $\A$, then $\sum_{p \in T_{= 1}}n_pm_p=nd$.

\item One has the inequality
$$(d-1)(d-2)  \geq  \sum_{p\in C \cap T_{= 1}} (m_p-1)^2.$$

\end{enumerate}

\end{proposition}

\proof To prove the first claim, note that the inclusion
$$C\cap\left(\bigcup_{H_i\in\A}H_i\right) \subset T_{= 1}$$
is obvious. Then
at each point $p\in C$, any line $H_j$ passing through $p$ is not in the tangent cone of the germ $(C,p)$, since otherwise the germs $(C,p)$ and $(H_j,p)$ are not separated after one blow-up.  This implies the first two claims in (1).
Moreover, if there is a line $H_j$ containing only one point $p$ of $T_{= 1}$, then the multiplicity of $C$ at $p$ has to be equal to $d$, which is possible only if $C$ is a line, since $C$ is irreducible.
Using (1), by summation over all the lines $H_j$, one gets (2).

To prove the last claim, note that 
since $C$ is irreducible, one has 
$$\chi(C)=b_0(C)-b_1(C)+b_2(C) \leq 2.$$
 On the other hand one has
$$\chi(C)=2-(d-1)(d-2)+ \sum_q \mu(C,q),$$
where the sum is over all singular points $q$ of $C$. It follows that
$$(d-1)(d-2) \geq  \sum_q \mu(C,q) \geq  \sum_{p\in C \cap T_{= 1}} \mu(C,p) \geq  \sum_{p\in C \cap T_{= 1}} (m_p-1)^2.$$
Indeed, the singularity $(C,p)$ has multiplicity $m_p$ and its Milnor number is at least $(m_p-1)^2$, by the well known semi-continuity property of the Milnor number.

\endproof

\begin{example} \label{Ziegler}
Let $p_1, \dots, p_6\in\PP^2$ be $6$ points such that no three are colinear. 
Let $\A=\{H_1, \dots, H_9\}$ be the edges of the corresponding hexagon and three diagonals (more precisely, the lines $p_1p_2, p_2p_3, p_3p_4, p_4p_5, p_5p_6, p_6p_1, p_1p_4, p_2p_5$ and $p_3p_6$), such that each line $H_j$ contains exactly 2 triple points and  4 nodes. 
Take $t_1, \dots t_9 \in\C^\times$ to be generic such that $\prod_{j=1}^9t_j=1$ and $t_{p_i}=1$ at 
all the vertices of the hexagon. Then note that $T_{\neq 1}=\emptyset$. 
Let $
D=\sum_{i=1}^9\overline{H}_i.%+\sum_{p\in T_{\neq 1}}E_p. 
$
Note that $D$ is connected and $D\cdot \overline{H}_i=1-2+0+4>0$. 
The ampleness of $D$  depends on the position of the six points $p_1, \dots, p_6$. 
Indeed, if $p_1, \dots, p_6$ are lying on a conic $C$, then the  strict 
transform $\overline{C}\subset S$ does not intersect $D$, and therefore $D\cdot \overline{C}=0$. 
Hence, when $p_1, \dots, p_6$ are  lying on a conic, then $S\smallsetminus D$ is not
affine.

We show  now that  when $p_1,....,p_6$ are not on a conic, there is no bad curve $C$ as above. Assume that in the real picture $p_1$, $p_2$, $p_3$, $p_4$ and $p_5$ are in this order on a conic. Denote $m_j=m_{p_j}$ and set $a=m_1$, $b=m_2$. It follows that $m_3=a$, $m_4=b$. The irreducible curve $C$ has degree $d=a+b$ by Proposition \ref{keyprop} (1). Hence, if $a>b$, the line $L$ determined by $p_1$ and $p_3$  satisfies $(L,C) \geq 2a >d$, and hence $C=L$, which is clearly not possible ($L$ intersects the arrangement in points not in $T_{= 1}$).
Hence $a=b$ and $d=2a$. There are two cases to consider.

\medskip

\noindent {\bf Case 1.} The curve $C$ passes through all the 6 points in $T_{= 1}$. 
Then the inequality in Proposition \ref{keyprop} (3)  becomes
$$(2a-1)(2a-2) \geq 6(a-1)^2,$$
which is equivalent to $(a-1)(a-2) \leq 0$. 

The case $a=1$ is impossible, since the 6 points are not on a conic. 
The case $a=2$ is also impossible, since this would imply the existence of an irreducible quartic curve with 6 points of multiplicity 2. By the genus formula involving the $\delta$-invariants, this is not possible.

\medskip

\noindent {\bf Case 2.} The curve $C$ meets a line $H_j$ just in one point. Then Proposition \ref{keyprop} (1) implies that $C$ is a line. Since any line through a triple point meets at least one line in $\A$ transversally in a node, the claim is proved.

Denote by $\A$ (resp. by $\A'$) the corresponding line arrangement when the 6 vertices of the hexagon are (resp. are not) on a conic. Note that $\A$ and $\A'$ are lattice-isotopic, if necessary refer to \cite{OT,DHA} for this notion, and hence the homeomorphism
$\phi : M(\A') \to M(\A)$ can be used to associate to any rank one local system $\LL$ on $M(\A)$
a rank one local system $\LL'=\phi^*(\LL)$ on $M(\A')$. These two local systems have the same monodromy with respect to the lines corresponding to each other under $\phi$, and moreover
$$H^1(M(\A), \LL)=H^1(M(\A'), \LL').$$
In other words, in spite of the fact that $M(\A')$ is affine and $M(\A)$ is not affine, their twisted cohomology with rank one coefficients are practically the same.

\end{example}

\section{A vanishing result}
We prove now the following result, generalizing the main result in \cite{cdo}.
See also the proof of \cite[Theorem 6.4.13]{SIT}, which is very similar to the proof below.

\begin{theorem}
\label{thm:vanish}
Let $S$ be a smooth proper complex variety of dimension $N \geq 2$. 
Let $D=\sum_{i=1}^nD_i$ be a divisor ($D_i$ irreducible). 
Let $I\subset [n]=\{1, \dots, n\}$ be a subset, and 
consider the corresponding subdivisor $D'=\sum_{i\in I}D_i$. 
Let $\LL$ be a rank one local system on $U=S\smallsetminus D$. Denote 
the monodromy of $\LL$ around $D_i$ by $t_i\in\C^\times$. 
Assume that 
\begin{itemize}
\item[(i)] 
$D$ is normal crossing along $D'$, that is, for any $p\in D'$, $D$ is 
normal crossing around $p$. 
\item[(ii)] 
$t_i\neq 1$ for $i\in I$. (In other words, $\LL$ has non-trivial monodromy 
around each component of $D'$.) 
\item[(iii)] 
$U'=S\smallsetminus D'$ is an affine variety. 
\end{itemize}
Then $H^1(S\smallsetminus D, \LL)=0$. 
\end{theorem}

\proof First note that we can assume $N=2$. Indeed, embed $S$ in some projective space $\PP^M$, and take $E \subset \PP^M$ a generic linear subspace of codimension $N-2$. Then $S \cap E$ is a smooth surface and by Zariski Theorem of Lefschetz type the inclusion $U \cap E=(S \cap E) \setminus (D \cap E) \to U$ is a 2-equivalence, see for instance \cite[Theorem (1.6.5)]{D1}. In other words, we can regard $U$ as being obtained from $U \cap E$ by adding cells of dimensions $\geq 3$, and 
 hence the above inclusion induces an isomorphism $H^1(U \cap E , \LL)=H^1(U,\LL).$

Assume from now on that $N=2$.
Let $i:U \to U'$ and $j:U' \to S$ be the two inclusions. The shifted sheaf $\LL[2]$ is a perverse sheaf on $U$. By Nagata's Theorem \ref{Nagata} and assumption (iii),
the inclusion $i$ is a quasi-finite affine morphism, so it follows by \cite[Corollary 5.2.17]{SIT} that $\F=Ri_*(\LL[2])$ is a perverse sheaf on $U'$. Since $U'$ is affine, one can use \cite[Corollary 5.2.19]{SIT} and conclude that $\HH^k(U',\F)=0$ for $k>0$ and $\HH^k_c(U',\F)=0$ for $k<0$. Let $a:S \to pt$ be the constant map to a point and recall that
$$\HH^k(U',\F)=H^k(Ra_*Rj_*\F)=H^k(Ra_*Rj_*Ri_*\LL[2])=H^{k+2}(U,\LL),$$
and also $\HH^k_c(U',\F)=H^k(Ra_!Rj_!\F)$. But $a$ is a proper map, and hence $Ra_*=Ra_!$. On the other hand, by the definition of the divisor $D'$, it follows that $Rj_!\F=Rj_*\F$. Indeed, the monodromy of the local system $\LL$ about an exceptional component $E_p$ for $p \in  T_{\ne1}$ is non-trivial, and hence the local twisted cohomology groups vanish. {\it The normal crossing property of $D$ along $D'$ is used to prove this local vanishing.}
It follows that $\HH^k(U',\F)=\HH^k_c(U',\F)$, and in particular 
$$ H^{1}(U,\LL)= \HH^{-1}(U',\F)=\HH^{-1}_c(U',\F)=0.$$
\endproof

\begin{example} \label{empty}
Let $\A$ be a line arrangement in $\PP^2$.
Suppose there exists a line $H\in \A$ such that $H\cap T_{=1}=\emptyset$. Then take $S$ to be the surface obtained from $\PP^2$ by blowing up only the points in $T\cap H$. Let $D$ be the reduced total transform of the union of all the lines in $\A$.
Then $S=\Bl_{T\cap H}\PP^2$ and $D'=\overline{H}+\sum_{p\in T\cap H}E_p$ 
satisfy the conditions of Theorem \ref{thm:vanish}. By the definition of $D'$, one sees that $U'=\C^2$ in this case.
This situation is 
nothing else but the result in \cite{cdo}. 

\end{example}

The main application of Theorem \ref{thm:vanish} is the following result, which was known to hold for real arrangements, see \cite{yos-mil, yos-locsys} and Remark \ref{rk:general} below.

\begin{proposition}
\label{prop:general} Let $\A$ be a line arrangement in $\PP^2$ and $\LL$ a rank one local system on the complement $M(\A)$.
Suppose that there exists a line $H_0\in\A$ such that 
$H_0\cap T_{=1}=\{p\}$ is a single point. Furthermore, assume that there is no line $L$ in $\PP^2$, passing through the point $p$ and 
such that $L\cap \A \subset T_{=1}$. 
Then $H^1(M(\A), \LL)=0$. 
\end{proposition}

\proof
Here we take $S$ to be the surface obtained from $\PP^2$ by blowing up all the points in $T$, and let $D$ be the reduced total transform of the union of all the lines in $\A$. In this way $U=M(\A)$.
Set $T_{\neq 1}\cap H_0=\{q_1, \hdots, q_s\}$. 

We first assume $s>0$. 
Denote by $n_i$ the 
multiplicity of $q_i$ in $\A$. 
For each $q_i$, 
set $\A_{q_i}=\{H_{i,1}, H_{i, 2}, \dots, H_{i, n_i-1}, H_0\}$. 
Consider the effective divisor $\widetilde{D}$ of the form 
\[
\widetilde{D}=
\overline{H}_0+
\sum_{i=1}^s b_i E_{q_i}+
\sum_{i=1}^s\sum_{c=1}^{n_i-1}a_{i,c}\overline{H}_{i,c}, 
\]
where $b_i, a_{i, c}>0$ are positive numbers which satisfy 
\begin{itemize}
\item 
$a_{i, c}<\frac{1}{- \overline{H}_{i, c}^2}$, if $\overline{H}_{i, c}^2<0$, 
\item 
$a_{i, c}>0$ is arbitrary, if $\overline{H}_{i, c}^2\geq 0$, 
\item 
$b_i$ satisfies $1<b_i<1+\sum_{c=1}^{n_i-1}a_{i, c}$. 
\end{itemize}
Let us prove that with these choices
$\widetilde{D}$ is an NM divisor. Indeed, one clearly has
\[
\begin{split}
\widetilde{D}\cdot\overline{H}_0
&=
1-1-s+\sum_{i=1}^s b_i=
\sum_{i=1}^s(b_i-1)>0. \\
\widetilde{D}\cdot\overline{E}_{q_i}
&=-b_i+1+\sum_{c=1}^{n_i-1}a_{i, c}>0\\
\widetilde{D}\cdot\overline{H}_{i, c}
&=b_i+a_{i,c}\overline{H}_{i, c}^2>b_i-1>0. 
\end{split}
\]
Therefore $\widetilde{D}$ is an NM divisor. 
Using Corollary \ref{cor:connect}, the divisor 
\[
D'=\overline{H}_0+\sum_{p\notin H}\overline{H}+\sum_{q\in T_{\neq 1}}E_q
\]
supports an NM divisor. Let $\widetilde{C}\subset S$ be an irreducible 
curve different from any components of ${D}'$ 
with $\widetilde{C}\cdot {D}'=0$. Then the image $C\subset\PP^2$ of $\widetilde{C}$ under the canonical projection $S \to \PP^2$
is a curve which intersects $H_0$ only at $p$. By Proposition \ref{keyprop} (1), $C$ is a line. 
However, by the assumption, 
$\widetilde{C}$ must intersect ${D'}$, which is a contradiction. 
Now it follows from Nakai-Moishezon criterion that $U'=S\smallsetminus D'$ is 
affine, and hence by Theorem \ref{thm:vanish} we get $H^1(M, \LL)=0$. 

Next we consider the case $s=0$. Since $\A$ is not a pencil type, 
there is $H_1$ that does not pass through $p$. Then it is easily seen that 
$\varepsilon \overline{H}_1+\overline{H}_0$ is an NM divisor for 
$0<\varepsilon\ll 1$. The remaining part is similar to the previous case. 
\endproof

\begin{remark}\label{rk:general}
It is natural to expect that 
Proposition \ref{prop:general} holds without assuming 
that {\it there is no line $L$ in $\PP^2$, passing through the point $p$ and 
such that $L\cap \A\subset T_{=1}$. }
Indeed, the claim without this extra condition is true for line arrangements defined over real numbers, see \cite{yos-mil, yos-locsys}. 
\end{remark}

The following result was established in \cite{Dreflection} using completely different techniques, namely residues of rational differential forms with poles along the line arrangement $\A$.

\begin{corollary}\label{cor:Ceva}

The monodromy eigenspaces $H^1(F,\C)_{\lambda}$ of the Milnor 
fiber cohomology  for the monomial arrangement a.k.a. the Ceva arrangement 
$$\A(m,m,3): (x^m-y^m)(y^m-z^m)(x^m-z^m)=0$$
are trivial if $\lambda^3 \ne 1$, for any $m \geq 3$.

\end{corollary}
\begin{proof} Any line in $\A(m,m,3)$ contains a point of multiplicity $m$ and $m$ triple points. It follows that $H^1(F,\C)_{\lambda}\ne 0 $ and $\lambda^3 \ne 1$ imply $\lambda^m=1$ by Libgober's result in \cite{Li}.
It is easy to see that the monomial arrangement satisfies the 
assumptions of Proposition \ref{prop:general} for the rank one local system $\LL$ associated to such an eigenvalue $\lambda$. 
\end{proof}
Since it was shown in \cite{MP} that , for   $\lambda^3 = 1$ , $\lambda \ne 1$, one  has $\dim H^1(F,\C)_{\lambda}=1$ when $m$ is not divisible by 3, and $\dim H^1(F,\C)_{\lambda}=2$ when $m$ is  divisible by 3, it follows that in this case the monodromy action on $H^1(F,\C)$ is completely known.

\section{The  exceptional reflection arrangement $\A(G_{31})$}

Recall first the defining equation $f=0$ for the reflection arrangement $\A(G_{31})$ in $\C^4$. One has 
\begin{equation}
\label{e6}
f=xyzt(x^4-y^4)(x^4-z^4)(x^4-t^4)(y^4-z^4)(y^4-t^4)(z^4-t^4)
\end{equation}
$$((x-y)^2-(z+t)^2) ((x-y)^2-(z-t)^2) ((x+y)^2-(z+t)^2)((x+y)^2-(z-t)^2)$$
$$((x-y)^2+(z+t)^2)((x-y)^2+(z-t)^2)((x+y)^2+(z+t)^2)((x+y)^2+(z-t)^2)$$
$$((x-z)^2+(y+t)^2) ((x-z)^2+(y-t)^2)((x+z)^2+(y+t)^2)((x+z)^2+(y-t)^2)$$
$$((x-t)^2+(y+z)^2)((x-t)^2+(y-z)^2)((x+t)^2+(y+z)^2)((x+t)^2+(y-z)^2),$$
see \cite{DStG31}. The result we prove in this section is the following.

\begin{proposition} \label{G31}
Let $\A(G_{31})$ be the reflection arrangement in $\C^4$ corresponding to the exceptional group $G_{31}$, and let $F$ be the associated Milnor fiber. Then the monodromy action on $H^1(F,\C)$ is the identity.

\end{proposition}

\proof 
We denote by $\A$ the arrangement in $\PP^2$ obtained by taking first the intersection of 
$\A(G_{31})$ with a generic hyperplane $H$ in $\C^4$, and then considering the corresponding projective line arrangement in $\PP^2=\PP(H)$. It is easy to check that one can take
$$H: 2x+5y-9z-t=0.$$
Then each line in $\A$ contains 12 nodes, 16 triple points and 3 points of multiplicity 6, see 
\cite[Table C.12, p. 293]{OT}. On the other hand, it is known that the monodromy on $H^1(F,\C)$ has no eigenvalues of order 2 and 3, see \cite{MPP}. It remains to show that eigenvalues of order 6 are also impossible. 

So let $\LL$ be the local system associated to such an eigenvalue of order 6. It follows that the corresponding set $ T_{\neq 1}$ consists of the set of triple points, and the set $T_{=1}$ consists of the points of multiplicity 6.
In this case there are $n=60$ lines $H_i$ and we can consider the divisor
\[
D'=\sum_{i=1}^{60}\overline{H}_i+\sum_{p\in T_{\neq 1}}E_p. 
\]
Using Proposition \ref{propNM} (1) it follows that
$$D' \cdot \overline{H}_i=1-(16+3)+16+12=10>0,$$
for any $i$. If $p \in T_{\neq 1}$, then
$D' \cdot E_p= 3-1=2>0$ and if  $p \in T_{=1}$, then $D' \cdot E_p=6>0.$ Hence $D$ is an NM divisor.

Let $C$ be a curve in $\PP^2$ intersecting the arrangement only in the points in $T_{=1}$, and transversally at these points as explained in Proposition \ref{keyprop}. If we show that such a curve does not exist, the claim in Proposition  \ref{G31} follows from Theorem \ref{thm:vanish}.
There are two cases to consider.

\medskip

\noindent {\bf Case 1.} Assume that all the 30 points in $T_{=1}$ are situated on the curve $C$.
With the notation from  Proposition \ref{keyprop}, we have $n_p=6$, and hence
$$s:=\sum m_p=10d.$$
Hence $d=s/10$ and the inequality in Proposition \ref{keyprop} (3) yields
$$(s-10)(s-20) \geq 100(\sum m_p^2-2s+30).$$
Since there are 30 points of multiplicity 6, using the inequality $30 \sum m_p^2 \geq s^2$,
we get that $s \in (25,48)$, and hence the only possible values for $d$ are $d=3$ and $d=4$.

Consider first the case $d=3$. Then the following 10 planes are clearly edges in the intersection lattice $L(\A(G_{31}))$ with multiplicity 6:
$$(x=y=0); \ (x=z=0); \ (x=t=0); \ (y=z=0); \ (y=t=0); \ (z=t=0);$$
$$(x-y=z-t=0); \ (x-y=z+t=0); \ (x+y=z-t=0); \ (x+y=z+t=0).$$
Using the equation of the hyperplane $H$, we see that the following 10 points, corresponding in order to the above 10 planes, are points of multiplicity 6 in the line arrangement $\A$:
$$p_1=(0:0:1), p_2=(0:1:0), p_3=(0:9:5), p_4=(1:0:0), $$
$$p_5=(9:0:2),p_6=(5:-2:0), p_7=(10:10:7), p_8=(8:8:7),$$ 
$$ p_9=(-10:10:3) \text{ and }  p_{10}=(-8:8:3).$$
A direct computation shows that there is no cubic plane curve passing through these 10 points.
Next consider the case $d=4$. Then we have to consider the following 5 additional 
 edges in the intersection lattice $L(\A(G_{31}))$ with multiplicity 6:
$$(x-z=y-t=0); \ (x-z=y+t=0); \ (x+z=y-t=0); \ (x+z=y+t=0); $$
$$\ (x-t=y-z=0),$$
and the corresponding 5 additional points of multiplicity 6 in the line arrangement $\A$:
$$p_{11}=(4:7:4), p_{12}=( 6:7:6), p_{13}=(-4:11:4), p_{14}=(-6:11:6)$$
and $ p_{15}=(4:1:1).$
A direct computation shows that there is no quartic plane curve passing through all these 15 points.

\medskip

\noindent {\bf Case 2.} Assume now that there is at least a point $p \in T_{=1}$ with $m_p=0$, i.e. $p$ is not on the curve $C$. The lines in $\A$ passing through $p$,
say $H_i$ for $i=1,...,6$, should contain each two points $p'_i,q'_i$ in $C \cap  T_{=1}$. Indeed, otherwise $C$ must be a line $L$ by Proposition \ref{keyprop} (1). But this line $L$ should intersect the arrangement only in points of multiplicity 6, hence it should contain 10 such points. However such a line $L$ does not exist, since we have the following result.

\begin{lemma}
\label{lem:G31}
Besides the points $p_j$ for $j=1,...,15$ of multiplicity 6 introduced above, consider the remaining $15$ points of multiplicity 6:
$$p_{16}=(14:-1:1), p_{17}=( 4:3:3), p_{18}=(14:-3:3), $$
$$p_{19}=( 9-i:1+9i:2+5i), p_{20}=(9+i:-1+9i:2+5i),$$
$$ p_{21}=(9-i:-1-9i:2-5i), p_{22}=( 9+i:1-9i:2-5i),$$
$$ p_{23}=(5+i:-2+9i:-1+5i), p_{24}=(5-i:-2+9i:-1+5i).$$
$$ p_{25}=( 5+i:-2-9i:1-5i), p_{26}=(5-i:-2-9i:-1-5i), $$
$$ p_{27}=(-5+9i:2-i:1+2i), p_{28}=(5-9i:-2-i:1-2i).$$
$$ p_{29}=( 5+9i:-2+i:1+2i), p_{30}=(-5-9i:2+i:1-2i),$$
where $i^2=-1$.
Then there are exactly 60 lines containing 3 of these 30 points of multiplicity 6, which are in fact the 60 lines of the line arrangement $\A$. More precisely the points $\{p_i,p_j,p_k\}$ are collinear if and only if the index set $\{i,j,k\}$ is one of the following 60 subsets of $\{1,2,...,30\}$.
$$\{1, 2, 3 \},
\{ 1, 4, 5 \},
\{ 1, 7 ,8 \},
 \{1, 9 ,10 \},
\{ 1, 19, 20 \},
\{ 1, 21, 22 \},$$
$$\{ 2, 4, 6 \},
\{ 2 ,11, 12\}, 
\{ 2, 13, 14 \},
\{ 2, 23, 24 \},
\{ 2 ,25, 26 \},
\{ 3, 5, 6 \},$$
$$\{ 3, 15, 16\}, 
\{ 3 ,17, 18 \},
\{ 3, 27, 29 \},
\{ 3, 28, 30 \},
\{ 4 ,15, 17 \},
\{ 4 ,16 ,18 \},$$
$$\{ 4 ,27, 28 \},
\{ 4, 29, 30 \},
\{ 5 ,11, 13 \},
\{ 5, 12 ,14 \},
\{ 5, 23, 25 \},
\{ 5, 24, 26 \},$$
$$\{ 6 ,7, 9 \},
\{ 6, 8 ,10 \},
 \{6 ,19 ,21 \},
 \{6 ,20, 22 \},
 \{7, 11 ,18 \},
 \{7 ,14, 15 \},$$
$$ \{7 ,24, 30 \},
\{ 7, 25, 27 \},
\{ 8 ,12, 16 \},
 \{8 ,13, 17 \},
\{ 8, 23, 29 \},
\{ 8, 26 ,28 \},$$
$$\{ 9, 12, 17 \},
 \{9 ,13 ,16 \},
 \{9, 23 ,28 \},
 \{9, 26, 29 \},
 \{10 ,11, 15 \},
 \{10, 14 ,18 \},$$
$$ \{10, 24, 27 \},
 \{10, 25, 30 \},
 \{11, 20, 28 \},
\{ 11, 21, 29 \},
\{ 12 ,19 ,27 \},
 \{12 ,22, 30 \},$$
$$\{13 ,19 ,30 \},
 \{13, 22, 27 \},
\{ 14, 20, 29 \},
\{ 14, 21, 28 \},
 \{15 ,19, 26 \},
\{ 15, 22 ,23 \},$$
$$\{ 16 ,20, 24 \},
\{ 16, 21, 25 \},
 \{17, 20 ,25 \},
 \{17, 21, 24 \},
 \{18 ,19, 23 \},
 \{18, 22, 26 \}.$$

Moreover, there is no line $L$ containing 4 points out of the 30 points $p_j$, $j=1,30$ of multiplicity 6.
\end{lemma}

\proof The proof is a direct check-up using a SINGULAR code written for us by Gabriel Sticlaru. The idea is to test all the subsets of 3 elements in the set of 30 points of multiplicity 6, and decide how many among these ${30 \choose 3}= 4060$ subsets correspond to collinear triplets. The same test for the ${30 \choose 4}= 27,405$ subsets with 4 elements gives no subset.

\endproof

\begin{remark}
\label{rk:G31}
A more rapid way, yielding less complete information, to show the non-existence of a line $L$ containing 10 points  of multiplicity 6 is to show that there is no line containing 4 out of the 24 points $p_i$, $i=1,...,24$. This approach involves only ${24 \choose 4}= 10,626$ subsets to be checked.
\end{remark}

Let $a_i>0$ (resp. $b_i>0$) be the multiplicity of $C$ at the points $p'_i$ (resp. at $q'_i$) introduced above. Then we have
$a_i+b_i=d$ for  $i=1,...,6$. The inequality in Proposition \ref{keyprop} (3) can be rewritten as
$$d^2-3d+2 \geq \sum_{i=1}^6( a_i-1)^2+ \sum_{i=1}^6( b_i-1)^2 = \sum_{i=1}^6 (a_i^2+ b_i^2) -2 \sum_{i=1}^6(a_i+b_i)+12=$$
$$= \sum_{i=1}^6(a_i^2+b_i^2)-12d+12.$$
On the other hand it is known that
$$12\sum_{i=1}^6 (a_i^2 +b_i^2) \geq \left( \sum_{i=1}^6 (a_i+b_i)\right)^2=36d^2.$$
with equality if and only if all $a_i, b_j$'s coincide. This leads to the inequality
$$2d^2-9d+10 \leq 0,$$
which is satisfied only for $d=2$. 

A subset of 12 points $P$ out of the 30 points $p_j$, $j=1,30$ of multiplicity 6 is called a {\it star configuration} if the points in $P$ can be divided into 6 pairs
 $\{p'_i, q'_i\}$ such that the 6 lines $L_i$ determined by $p'_i,q'_i$ are in $\A$ and meet all at a point $p\in \A$ of multiplicity 6. The point $p$ is called in this situation the {\it center} of the configuration $P$.

\begin{lemma}
\label{lem:G31bis}

There are 10 conics in $\PP^2$ containing exactly 12 points out of the 30 points $p_j$, $j=1,30$ of multiplicity 6, and the list of the corresponding subsets of indices of the 12 points on a conic is given below. 
$$\{ 1, 2, 5, 6, 7, 10, 11, 14, 20, 21, 24, 25\}, \  \{ 1, 2, 5, 6, 8, 9, 12, 13, 19, 22, 23, 26\},$$
 $$\{1, 3, 4, 6, 7 ,10, 15, 18, 19, 22, 27, 30\}, \ \{1, 3, 4, 6, 8, 9, 16, 17, 20, 21, 28, 29\},$$
 $$\{2, 3, 4, 5, 11, 14, 15, 18, 23, 26, 28, 29\}, \  \{ 2, 3, 4, 5, 12, 13, 16, 17, 24, 25, 27, 30\},$$
 $$\{7, 8, 9, 10, 11, 12, 13, 14, 27, 28, 29, 30\}, \  \{ 7, 8, 9, 10, 15, 16, 17, 18, 23, 24, 25, 26\},$$
$$\{ 11, 12, 13, 14, 15, 16, 17, 18, 19, 20, 21, 22\},$$
 $$\{19, 20, 21, 22, 23, 24, 25, 26, 27, 28, 29, 30\}.$$
Any such subset of 12 points is not a star configuration.
\end{lemma}

\proof The proof that the above conics are exactly those containing 12 points is a direct check-up using the same SINGULAR code as in the proof of Lemma \ref{lem:G31}. It involves a check of all the 
$$N={30 \choose 12}=86,493,225$$
 subsets with 12 elements in the set of 30 points of multiplicity 6. Note that a conic cannot contain 13 points of multiplicity 6 in $\A$, since otherwise it would generate 13 subsets by deleting one point at a time. To show that none of the above 10 subsets of points is a star configuration is very easy using Lemma \ref{lem:G31}. For instance, suppose the first set is a star configuration. Because $1,2$ are elements of this set, the center should be $3$. But $1,5$ are also elements, so the center should be $4$, a contradiction. Note that these 10 subsets do not play a symmetric role: for instance the intersection of the first 3 subsets consists of two elements, namely $1,6$, while the intersection of the last 3 subsets is empty. However, any two subsets have exactly 4 elements in common.

\endproof

\begin{remark}
\label{rk:G31bis}
Note that a conic containing 12 points of multiplicity 6 must be irreducible, and hence smooth by
Lemma \ref{lem:G31}. But by Bezout Theorem, such a conic must meet the arrangement in exactly 20 points of multiplicity 6, according to Proposition \ref{keyprop}.
A more rapid way, yielding less complete information, to show the non-existence of such a conic is to show that there is no conic containing 13 out of the 23 points $p_i$, $i=1,...,23$. This approach involves only ${23 \choose 13}= 1,144,066$ subsets to be checked.
\end{remark}

\end{document}